\documentclass{amsart}

\usepackage{amssymb}

\newtheorem{theorem}{Theorem}[section]
\newtheorem{lemma}[theorem]{Lemma}
\newtheorem{proposition}[theorem]{Proposition}
\newtheorem{corollary}[theorem]{Corollary}
\theoremstyle{definition}
\newtheorem{definition}[theorem]{Definition}
\newtheorem{example}[theorem]{Example}
\newtheorem{remark}[theorem]{Remark}

\begin{document}

\title[Doubly indexed flag variety]
{
Doubly indexed flag variety and fixed point set of a partial flag variety}

\author{Lucas Fresse}
\thanks{Work supported in part by Minerva grant, No. 8596/1.}
\address{Department of Mathematics, the Weizmann Institute of Science, 76100 Rehovot, Israel}
\email{lucas.fresse@weizmann.ac.il}

\begin{abstract}
We define a variety of doubly indexed flags, this is a smooth, projective variety,
and we describe it as an iterated over Grassmannian varieties.
On the other hand, we consider the variety of partial flags which are stabilized
by a given nilpotent endomorphism.
We partition this variety into locally closed subvarieties
which are vector bundles over varieties of the aforedmentioned type.
\end{abstract}

\keywords{Flag manifolds; partial flags; parabolic orbits; iterated bundles; vector bundles; Springer fibers}
\subjclass[2010]{14M15; 14L30; 17B08}

\maketitle

\section{Introduction}

All along this article, we fix an algebraically closed field $\mathbb{K}$ of characteristic zero,
and this is implicitely the underlying field of all the constructions.
Let $A$ be a finite dimensional $\mathbb{K}$-vector space and let $x\in\mathrm{End}(A)$ be nilpotent.
In this article, 
we consider the variety ${\mathcal A}$ of partial flags
$(V_1\subset\ldots\subset V_n\subset A)$ of a given dimension vector
$(k_1\leq \ldots\leq k_n)$, and we study 
the subvariety ${\mathcal A}_x$ formed by $x$-stable partial flags,
i.e. such that $x(V_i)\subset V_i$ for all $i$.
In the particular case where ${\mathcal A}$ is a variety of complete flags,
${\mathcal A}_x$ is a Springer fiber.

In general, the variety ${\mathcal A}_x$ is a connected, algebraic projective variety, 
but outside of trivial cases it is reducible and singular (cf. \cite{Spaltenstein-1982}, \cite{Steinberg-1976}).

In \cite{Spaltenstein-1976}, N. Spaltenstein constructed a partition of ${\mathcal A}_x$
into locally closed, smooth, irreducible subvarieties.
Moreover, it is known (cf. \cite{Shimomura}) that the variety ${\mathcal A}_x$ admits a
cell decomposition.

Also in this article, we study the problem of partitioning ${\mathcal A}_x$ into a finite number
of smooth, irreducible subvarieties. Let us outline our construction.
In section \ref{section-doubly-indexed-flags}, we define a variety
${\mathcal F}_{\kappa_{\bullet,\bullet}}(A_\bullet)$ of doubly indexed sequences of nested
subspaces $(V_{i,j})$, we show that it is projective and smooth, and in fact that it is
an iterated bundle of Grassmannian varieties (Proposition \ref{proposition-iterated-bundle}). 
The variety ${\mathcal F}_{\kappa_{\bullet,\bullet}}(A_\bullet)$ is also a natural desingularization
of a Schubert variety (Corollary \ref{corollary-resolution-singularity}).
In section \ref{section-partition}, we consider the action on the variety ${\mathcal A}$
of the parabolic group $P\subset GL(A)$ of elements which fix the kernels $\ker x^i$.
We partition ${\mathcal A}_x$ into its intersections with the $P$-orbits of ${\mathcal A}$.
We show that each intersection ${\mathcal A}_x\cap {\mathcal P}$ is a vector bundle over
a variety of the form ${\mathcal F}_{\kappa_{\bullet,\bullet}}(A_\bullet)$
(the main result of this paper, Theorem \ref{theorem1}).

Our construction can be related to
the construction by C. de Concini, G. Lusztig, C. Procesi
\cite{DeConcini-Lusztig-Procesi} of a partition of a general Springer fiber
into vector bundles over smooth, projective varieties,
also obtained by taking the
intersections with the orbits of a parabolic group attached to
the nilpotent element (cf. Remark \ref{remark2}).

\section{The variety ${\mathcal F}_{\kappa_{\bullet,\bullet}}(A_\bullet)$}

\label{section-doubly-indexed-flags}

Let $A$ be a vector space of finite dimension $d\geq 0$, and let $k\in\{0,1,\ldots,d\}$.
Let $Grass_k(A)$ be the variety of $k$-dimensional vector spaces $V\subset A$.
We will abbreviate $Grass_k(d)=Grass_k(\mathbb{K}^d)$.

More generally, for a uple
$k_\bullet=(k_1\leq k_2\leq\ldots\leq k_n)$ with $k_i\in\{0,1,\ldots,d\}$,
we denote by ${\mathcal F}_{k_\bullet}(A)$ the variety of partial flags
$(V_1\subset V_2\subset \ldots\subset V_n)$ with dimension vector $k_\bullet$,
that is $V_i\in Grass_{k_i}(A)$ for all $i$.

We introduce a variety of doubly indexed partial flags which generalizes ${\mathcal F}_{k_\bullet}(A)$:

\begin{definition}
\label{definition-1}
Fix a partial flag $A_\bullet=(A_1\subset\ldots\subset A_m)$
with dimension vector $d_\bullet=(d_1,\ldots,d_m)$, and fix a matrix
$\kappa_{\bullet,\bullet}=(\kappa_{i,j})_{1\leq i\leq m,\,1\leq j\leq n}$
where $\kappa_{i,j}\in\{0,1,\ldots,d_i\}$, and with $\kappa_{i,j}\leq \kappa_{i',j'}$
whenever $i\leq i'$ and $j\leq j'$.
We define ${\mathcal F}_{\kappa_{\bullet,\bullet}}(A_\bullet)$ as the set of doubly indexed
sequences $(V_{i,j})_{1\leq i\leq m,\,1\leq j\leq n}$ where
$V_{i,j}\in Grass_{\kappa_{i,j}}(A_i)$, and with $V_{i,j}\subset V_{i',j'}$ whenever $i\leq i'$ and $j\leq j'$.
\end{definition}

An element in ${\mathcal F}_{\kappa_{\bullet,\bullet}}(A_\bullet)$ is thus an arrangement of subspaces
\[
\begin{array}{cccccccccccccccccccccccccccccc}
V_{1,1} & \subset & V_{1,2} & \subset & \ldots & \subset & V_{1,n} & \subset & A_1 \\
\cap & & \cap & & & & \cap & & \cap \\
V_{2,1} & \subset & V_{2,2} & \subset & \ldots & \subset & V_{2,n} & \subset & A_2 \\
\cap & & \cap & & & & \cap & & \cap \\
\vdots & & \vdots & & & & \vdots & & \vdots \\
\cap & & \cap & & & & \cap & & \cap \\
V_{m,1} & \subset & V_{m,2} & \subset & \ldots & \subset & V_{m,n} & \subset & A_m
\end{array}
\]
The set ${\mathcal F}_{\kappa_{\bullet,\bullet}}(A_\bullet)$ is a closed subvariety of the projective variety
$\prod_{i,j}Grass_{\kappa_{i,j}}(A_i)$.
We will abbreviate ${\mathcal F}_{\kappa_{\bullet,\bullet}}(d_\bullet)={\mathcal F}_{\kappa_{\bullet,\bullet}}(\mathbb{K}^{d_1}\subset
\ldots\subset \mathbb{K}^{d_m})$.

\begin{example}
If $m=1$, then the sequence $A_\bullet=(A_1)$ is a single space, the matrix $\kappa_{\bullet,\bullet}=(\kappa_{1,1},\ldots,\kappa_{1,n})=:k_\bullet$
is a $n$-uple, and in these terms the variety ${\mathcal F}_{\kappa_{\bullet,\bullet}}(A_\bullet)$ coincides
with the partial flag variety ${\mathcal F}_{k_\bullet}(A_1)$.
\end{example}

\begin{example}
\label{example-2}
Suppose $n=1$. Thus $\kappa_{\bullet,\bullet}=(\kappa_{1,1},\ldots,\kappa_{m,1})=:(k_1,\ldots,k_m)$
is a $m$-uple. 
In this case,
the variety ${\mathcal F}_{\kappa_{\bullet,\bullet}}(A_\bullet)$ coincides with the Schubert variety
\[
{\mathcal X}_{k_\bullet}(A_\bullet):=
\{(V_1\subset\ldots\subset V_m)\in{\mathcal F}_{k_\bullet}(A_m):V_i\subset A_i\ \forall i\}.
\]
We will also abbreviate ${\mathcal X}_{k_\bullet}(d_\bullet)={\mathcal X}_{k_\bullet}(\mathbb{K}^{d_1}\subset
\ldots\subset \mathbb{K}^{d_m})$.
\end{example}

Our next purpose is to describe the structure of the variety ${\mathcal F}_{\kappa_{\bullet,\bullet}}(A_\bullet)$.
We recall the notion of iterated (fiber) bundle.
Let $X,B_1,\ldots,B_r$ be algebraic varieties.
For $r=1$, we say that $X$ is an iterated bundle of base $B_1$ if there is an isomorphism $X\stackrel{\sim}{\rightarrow} B_1$.
For $r\geq 2$, we say that $X$ is an iterated bundle of base $(B_1,\ldots,B_r)$ if there is a locally
trivial fiber bundle $X\rightarrow B_r$ whose typical fiber is an iterated bundle of base $(B_1,\ldots,B_{r-1})$.
An immediate example of iterated bundle is a product $X=B_1\times\ldots\times B_r$.
Another example: the partial flag variety ${\mathcal F}_{(k_1,\ldots,k_n)}(A)$ 
is an iterated bundle of base $\big(Grass_{k_1}(k_2),\ldots,Grass_{k_{n-1}}(k_n),Grass_{k_n}(A)\big)$.

In the setting of Example \ref{example-2},
the map ${\mathcal X}_{k_\bullet}(A_\bullet)\rightarrow Grass_{k_1}(A_1)$,
$(V_1,\ldots,V_m)\mapsto V_1$ is a locally trivial fiber bundle
of typical fiber ${\mathcal X}_{k_2-k_1,\ldots,k_m-k_1}(A_2/V_1,\ldots,A_m/V_1)$
hence, by induction,
we infer that ${\mathcal X}_{k_\bullet}(A_\bullet)$ is an iterated bundle of base
\[
\big(Grass_{k_m-k_{m-1}}(d_m-k_{m-1}),\ldots,Grass_{k_2-k_1}(d_2-k_1),Grass_{k_1}(d_1)\big).
\]
Thus, in the situation $n=1$, we get that
${\mathcal F}_{\kappa_{\bullet,\bullet}}(A_\bullet)$ is an iterated bundle of base
a sequence of Grassmannian varieties. 

Let us generalize this fact.
Let us come back to the general variety ${\mathcal F}_{\kappa_{\bullet,\bullet}}(A_\bullet)$,
with the same notation as in Definition \ref{definition-1}.
For $j\in \{1,\ldots,n\}$, we write $\kappa_{\bullet,j}=(\kappa_{1,j},\ldots,\kappa_{m,j})$.

\begin{proposition}
\label{proposition-iterated-bundle}
(a) The variety ${\mathcal F}_{\kappa_{\bullet,\bullet}}(A_\bullet)$
is an iterated bundle of base
\[
\big(
{\mathcal X}_{\kappa_{\bullet,1}}(\kappa_{\bullet,2}),
\ldots,{\mathcal X}_{\kappa_{\bullet,n-1}}(\kappa_{\bullet,n}),
{\mathcal X}_{\kappa_{\bullet,n}}(d_\bullet)
\big).
\]
(b) In particular, ${\mathcal F}_{\kappa_{\bullet,\bullet}}(A_\bullet)$
is an iterated bundle of base a sequence of Grassmannian varieties.
More precisely,
define $m\times n$-matrices $\delta_{\bullet,\bullet}$ and $\varepsilon_{\bullet,\bullet}$
by setting 
\[
\delta_{i,j}=\left\{
\begin{array}{ll}
\kappa_{i,j}-\kappa_{i-1,j} & \mbox{if $i\geq 2$,} \\
\kappa_{1,j} & \mbox{if $i=1$;} \\
\end{array}
\right.
\qquad
\varepsilon_{i,j}=\left\{
\begin{array}{ll}
\kappa_{i,j+1}-\kappa_{i,j} & \mbox{if $j\leq n-1$,} \\
d_i-\kappa_{i,n} & \mbox{if $j=n$;} \\
\end{array}
\right.
\]
and set ${\mathcal G}_{i,j}=Grass_{\delta_{i,j}}(\delta_{i,j}+\varepsilon_{i,j})$.
Then, ${\mathcal F}_{\kappa_{\bullet,\bullet}}(A_\bullet)$ is an iterated bundle
of base
\[\big(
{\mathcal G}_{m,1},\ldots,{\mathcal G}_{1,1},{\mathcal G}_{m,2},\ldots,{\mathcal G}_{1,2},
\ldots,{\mathcal G}_{m,n},\ldots,{\mathcal G}_{1,n}
\big).\]
(c) In particular, the variety ${\mathcal F}_{\kappa_{\bullet,\bullet}}(A_\bullet)$ is smooth.
\end{proposition}

\begin{proof}
First, note that (b) is a consequence of (a) together with the above description
of ${\mathcal X}_{k_\bullet}(A_\bullet)$.
Also, (c) is a consequence of (b), as well as of (a).
Thus, it remains to establish (a). We reason by induction on $n\geq 1$,
with immediate initialization for $n=1$. Assume the property holds until $n-1\geq 1$. Consider the map
\[
\varphi:{\mathcal F}_{\kappa_{\bullet,\bullet}}(A_\bullet) \rightarrow {\mathcal X}_{\kappa_{\bullet,n}}(A_\bullet),
\ (V_{i,j})\mapsto (V_{1,n},\ldots,V_{m,n}).
\]
The fiber of $\varphi$ over $W_\bullet=(W_1,\ldots,W_m)\in {\mathcal X}_{\kappa_{\bullet,n}}(A_\bullet)$
is isomorphic to ${\mathcal F}_{\tilde\kappa_{\bullet,\bullet}}(W_\bullet)$
for $\tilde\kappa_{\bullet,\bullet}=(\kappa_{i,j})_{1\leq i\leq m,\,1\leq j\leq n-1}$.
Therefore, the induction hypothesis applies to the fiber $\varphi^{-1}(W_\bullet)$. 
It therefore remains to show that $\varphi$ is locally trivial.

Let $P\subset GL(A_m)$ be the parabolic subgroup of elements which stabilize the partial flag $A_\bullet$.
Thus, ${\mathcal X}_{\kappa_{\bullet,n}}(A_\bullet)$ is a $P$-orbit of the partial flag variety ${\mathcal F}_{\bullet,n}(A_m)$.
Also, note that $P$ naturally acts on ${\mathcal F}_{\kappa_{\bullet,\bullet}}(A_\bullet)$,
and the map $\varphi$ is $P$-equivariant.
Fix $W_\bullet\in {\mathcal X}_{\kappa_{\bullet,n}}(A_\bullet)$.
By Schubert decomposition, there is a unipotent subgroup $U\subset P$ such that
the map $\rho:U\rightarrow {\mathcal X}_{\kappa_{\bullet,n}}(A_\bullet)$, $g\mapsto gW_\bullet$
is an open immersion (see \cite[\S 1.2]{Brion}).
We get an isomorphism
$U\times \varphi^{-1}(W_\bullet)\stackrel{\sim}{\rightarrow} \varphi^{-1}(\rho(U))$, $(g,(V_{i,j}))\mapsto (gV_{i,j})$.
It results that the map $\varphi$ is trivial over the open set
$r(U)\subset {\mathcal X}_{\kappa_{\bullet,n}}(A_\bullet)$.
\end{proof}

For $i\in \{1,\ldots,m\}$, we write $\kappa_{i,\bullet}=(\kappa_{i,1},\ldots,\kappa_{i,n})$.
From the definition of the variety ${\mathcal F}_{\kappa_{\bullet,\bullet}}(A_\bullet)$
and the fact that it is smooth,
we derive:

\begin{corollary}
\label{corollary-resolution-singularity}
The map ${\mathcal F}_{\kappa_{\bullet,\bullet}}(A_\bullet)\rightarrow 
{\mathcal F}_{\kappa_{m,\bullet}}(A_m)$
$(V_{i,j})\mapsto (V_{m,1},\ldots,V_{m,n})$
is a resolution of singularity of the Schubert variety
\[
\{(V_1\subset\ldots\subset V_n)\in {\mathcal F}_{\kappa_{m,\bullet}}(A_m):\dim A_i\cap V_j\geq \kappa_{i,j}\ \forall i,j\}.
\]
\end{corollary}

\section{Partition of the variety of $x$-stable partial flags}

\label{section-partition}

In this section, we fix a vector space $A$ of finite dimension $d\geq 0$
and a dimension vector $k_\bullet=(k_1,\ldots,k_n)$ with $k_i\in\{0,1,\ldots,d\}$.
We consider the partial flag variety ${\mathcal F}_{k_\bullet}(A)$ which, for simplicity, we denote by ${\mathcal A}$.
We fix a nilpotent endomorphism $x\in\mathrm{End}(A)$ and we study the subvariety
${\mathcal A}_x\subset {\mathcal A}$ formed by the $x$-stable elements,
i.e. the flags $(V_1,\ldots,V_n)\in{\mathcal A}$ such that $x(V_i)\subset V_i$ for all $i$.
Our purpose is to describe ${\mathcal A}_x$ as a union of locally closed subvarieties
which are vector bundles over varieties of the form ${\mathcal F}_{\kappa_{\bullet,\bullet}}(A_\bullet)$.
Let us formulate our result.

Let $m\geq 0$ be the nilpotent order of $x$. Consider the partial flag
\[
\ker x\subset\ldots\subset \ker x^{m-1}\subset \ker x^m=A,
\]
and let $P\subset GL(A)$ be the parabolic subgroup formed by the elements which leave it invariant.
Choose a semisimple element $s\in P$ whose centralizer $Z_G(s)$ is a Levi subgroup of $P$
and such that $sxs^{-1}=qx$ for some $q\in\mathbb{K}^*$.
Thus, the action of $s$ on ${\mathcal A}$ restricts to an action on ${\mathcal A}_x$,
and we denote by ${\mathcal A}_x^s\subset {\mathcal A}_x$ the subset of fixed points for this action.

\begin{theorem}
\label{theorem1}
Let ${\mathcal P}\subset {\mathcal A}$ be a $P$-orbit. \\
(a) We have ${\mathcal A}_x\cap {\mathcal P}\not=\emptyset$ if and only if 
${\mathcal A}^s_x\cap {\mathcal P}\not=\emptyset$. In what follows, suppose
that both properties are satisfied. \\
(b) The subset ${\mathcal A}^s_x\cap {\mathcal P}$ is open and closed in ${\mathcal A}_x^s$,
and it is isomorphic to a variety of the form ${\mathcal F}_{\kappa_{\bullet,\bullet}}(A_\bullet)$. \\
(c) There is an algebraic vector bundle
${\mathcal A}_x\cap {\mathcal P}\rightarrow {\mathcal A}^s_x\cap {\mathcal P}$.
\end{theorem}

The remainder of the section is devoted to the proof of the theorem.
The proof will include complementary details, for instance we will explicit the matrix 
$\kappa_{\bullet,\bullet}$ and the partial flag $A_\bullet$ in (b)
(cf. Proposition \ref{proposition-2}).

As a preliminary step, let us give another description of the subset ${\mathcal A}^s\subset {\mathcal A}$
of $s$-fixed points.
From the choice of $s$, 
there is a decomposition 
\[
A=E_1\oplus\ldots\oplus E_m
\]
such that $\ker x^i=E_1\oplus\ldots\oplus E_i$ for all $i$, 
with $x(E_i)\subset E_{i-1}$ for $i\geq 2$, and such that $s$ acts on $E_i$
by multiplication by $\lambda_0q^{-i}$ for some $\lambda_0\in\mathbb{K}^*$.
The eigenvalues $\lambda_0q^{-i}$, $i=1,\ldots,m$, are necessarily pairwise distinct. 
Note that in particular, $A=E_1\oplus\ldots\oplus E_m$ is a graduated vector space.
Then, ${\mathcal A}^s$ is the subset of flags $(V_1,\ldots,V_n)\in{\mathcal A}$ which are
homogeneous with respect to this graduation, i.e. such that
$V_j=(E_1\cap V_j)\oplus\ldots\oplus (E_m\cap V_j)$ for all $j$.

Recall that a $P$-orbit in the partial flag variety ${\mathcal A}$ is of the form
${\mathcal P}={\mathcal P}_{\alpha_{\bullet,\bullet}}$ with
\[
{\mathcal P}_{\alpha_{\bullet,\bullet}}=\left\{(V_1,\ldots,V_n)\in{\mathcal A}:V_j\cap\dim \ker x^i=\alpha_{i,j}\ \forall i,j\in\{1,\ldots,m\}{\times}\{1,\ldots,n\}\right\}
\]
where $\alpha_{\bullet,\bullet}$ is a fixed matrix of integers.
Recall that ${\mathcal A}$ stands for the variety of partial flags $(V_1\subset\ldots\subset V_n\subset A)$
with $\dim V_j=k_j$ for all $j$.
In addition, let $l_i=\dim \ker x^i$.
From the definition of ${\mathcal P}={\mathcal P}_{\alpha_{\bullet,\bullet}}$, we get that, if the intersection
${\mathcal A}_x\cap {\mathcal P}$ is nonempty, then $\alpha_{\bullet,\bullet}$ satisfies
\begin{eqnarray}
\label{alpha1}
&
0\leq\alpha_{i,j}\leq \alpha_{i',j'}\mbox{ whenever $i\leq i'$, $j\leq j'$, and }\alpha_{m,j}=k_j\mbox{ for all $j$,}\\
\label{alpha2}
& \alpha_{i',j'}-\alpha_{i'-1,j'}\leq \alpha_{i,j}-\alpha_{i-1,j}\leq \alpha_{1,j}\mbox{ whenever $2\leq i\leq i'$, $j'\leq j$},\\
\label{alpha3}
& \alpha_{i,j}-\alpha_{i-1,j}\leq l_i-l_{i-1}\mbox{ for $2\leq i\leq m$},\ \mbox{and}\ \alpha_{1,j}\leq l_1\mbox{ for all $j$}. 
\end{eqnarray}
Indeed, (\ref{alpha1}) is necessary to guarantee that ${\mathcal P}$
is nonempty, whereas (\ref{alpha2}), (\ref{alpha3}) are implied by the fact that there are injective linear maps
$V_j\cap\ker x^{i'}/V_j\cap\ker x^{i'-1}\hookrightarrow V_j\cap\ker x^i/V_j\cap\ker x^{i-1}$ (induced by $x^{i'-i}$)
and $V_j\cap\ker x^i/V_j\cap\ker x^{i-1}\hookrightarrow\ker x^i/\ker x^{i-1}$
whenever there is an element $V_\bullet\in {\mathcal A}_x\cap {\mathcal P}$.
Conversely, the fact that $\alpha_{\bullet,\bullet}$ satisfies (\ref{alpha1}), (\ref{alpha2}), (\ref{alpha3})
is enough to ensure that the intersection ${\mathcal A}_x^s\cap {\mathcal P}$ is nonempty
(and from this it results point (a) of the theorem).
This is shown in particular by the following proposition, which also establishes point (b) of the theorem.

\begin{proposition}
\label{proposition-2}
Let the matrix $\alpha_{\bullet,\bullet}$ satisfy (\ref{alpha1}), (\ref{alpha2}), (\ref{alpha3}),
and let ${\mathcal P}={\mathcal P}_{\alpha_{\bullet,\bullet}}$.
Let $A_\bullet=(A_1,\ldots,A_m)$ be the partial flag defined by $A_i=x^{m-i}(\ker x^{m-i+1})$ for all $i$,
that is
\[
A_\bullet=(A_1,\ldots,A_m)=\left(x^{m-1}(\ker x^m)\subset \ldots\subset x(\ker x^2)\subset \ker x\right).
\]
Thus, $A_\bullet$ has dimension vector
$d_\bullet=(d_1,\ldots,d_m)$, with $d_i
=\dim \ker x^{m-i+1}-\dim \ker x^{m-i}$.
Let $\kappa_{\bullet,\bullet}=\left(\kappa_{i,j}\right)_{1\leq i\leq m,\,1\leq j\leq n}$
be the matrix with
\[
\kappa_{i,j}=\left\{\begin{array}{ll}
\alpha_{m-i+1,j}-\alpha_{m-i,j} & \mbox{if $i\leq m-1$,} \\
\alpha_{1,j} & \mbox{if $i=m$.}
\end{array}\right.
\]
(a) The intersection ${\mathcal A}_x^s\cap {\mathcal P}$ is open and closed in ${\mathcal A}_x^s$. \\
(b) The map
\[
{\mathcal A}_x^s\cap {\mathcal P}\rightarrow {\mathcal F}_{\kappa_{\bullet,\bullet}}(A_\bullet),\
(V_1,\ldots,V_n)\mapsto \left(\,x^{m-i}(V_j\cap \ker x^{m-i+1})\,\right)_{1\leq i\leq m,\,1\leq j\leq n}
\]
is well-defined and it is an isomorphism of algebraic varieties.
\end{proposition}

\begin{proof}
(a) We note that
\[
\begin{array}{l}
{\mathcal A}_x^s\cap {\mathcal P} \ =\ \{V_\bullet\in {\mathcal A}_x^s: \dim V_j\cap E_{m-i+1}=\kappa_{i,j}\ \forall i,j\} \\[2mm]
= \ \{V_\bullet\in {\mathcal A}_x^s: \dim V_j\cap E_{m-i+1}\leq \kappa_{i,j}\ \mbox{and}\ \dim V_j\cap \textstyle\!\!\!\bigoplus\limits_{h\not=m-i+1}\!\!\!\!E_h\leq k_j-\kappa_{i,j}\ \forall i,j\} \\[4mm]
= \ \{V_\bullet\in {\mathcal A}_x^s: \dim V_j\cap E_{m-i+1}\geq \kappa_{i,j}\ \mbox{and}\ \dim V_j\cap \textstyle\!\!\!\bigoplus\limits_{h\not=m-i+1}\!\!\!\!E_h\geq k_j-\kappa_{i,j}\ \forall i,j\},
\end{array}
\]
hence ${\mathcal A}_x^s\cap {\mathcal P}$ is open and closed in ${\mathcal A}_x^s$.

\smallskip
(b) We denote by $\varphi$ the map of the statement.
Observe that $x^{m-i}$ restricts to an isomorphism $x^{m-i}:E_{m-i+1}\stackrel{\sim}{\rightarrow}A_i$.
Thus, $\dim x^{m-i}(V_j\cap \ker x^{m-i+1})=\dim x^{m-i}(V_j\cap E_{m-i+1})=\kappa_{i,j}$.
In addition, for $i\leq i'$ and $j\leq j'$, we have
\begin{eqnarray*}
x^{m-i}(V_j\cap \ker x^{m-i+1}) & = & x^{m-i'}(x^{i'-i}(V_{j})\cap x^{i'-i}(\ker x^{m-i+1})) \\
 & \subset & x^{m-i'}(V_{j}\cap \ker x^{m-i'+1}) \\
 & \subset & x^{m-i'}(V_{j'}\cap \ker x^{m-i'+1}).
\end{eqnarray*}
Therefore, $\left(\,x^{m-i}(V_j\cap \ker x^{m-i+1})\,\right)_{i,j}\in {\mathcal F}_{\kappa_{\bullet,\bullet}}(A_\bullet)$,
so that $\varphi$ is well defined. It is clearly algebraic.

Let $\hat{x}^{m-i}:A_i\rightarrow E_{m-i+1}$ denote the inverse of $x^{m-i}$.
For $V_{\bullet,\bullet}=(V_{i,j})\in {\mathcal F}_{\kappa_{\bullet,\bullet}}(A_\bullet)$,
we set 
\[
\psi(V_{\bullet,\bullet})=(W_1,\ldots,W_n),\ \mbox{with }W_j=\textstyle\bigoplus\limits_{i=1}^m\hat{x}^{m-i}(V_{i,j})\mbox{ for all $j$}.
\]
First, we have $\dim W_j=\sum_{i=1}^m\kappa_{i,j}=\alpha_{m,j}=k_j$.
Second, if $j\leq j'$, then we have $\hat{x}^{m-i}(V_{i,j})\subset \hat{x}^{m-i}(V_{i,j'})$ for all $i$,
thus $W_j\subset W_{j'}$.
Hence, $\psi(V_{\bullet,\bullet})\in {\mathcal A}^s$.
Next, $x(V_{m,j})=0$ and, for $i\leq m-1$,
$x(\hat{x}^{m-i}(V_{i,j}))\subset \hat{x}^{m-i-1}(V_{i,j})\subset \hat{x}^{m-i-1}(V_{i+1,j})$,
thus $x(W_j)\subset W_j$ for all $j$, so $\psi(V_{\bullet,\bullet})\in {\mathcal A}^s_x$.
Finally, we see that $\dim W_j\cap \ker x^i=\sum_{h=1}^i\kappa_{h,j}=\alpha_{i,j}$.
Therefore, $\psi(V_{\bullet,\bullet})\in {\mathcal A}_x^s\cap {\mathcal P}$.
So, $\psi$ is well defined from ${\mathcal F}_{\kappa_{\bullet,\bullet}}(A_\bullet)$
to ${\mathcal A}_x^s\cap {\mathcal P}$.

Clearly, $\psi$ is algebraic, and we have $\psi\circ\varphi=\mathrm{id}_{{\mathcal A}_x^s\cap{\mathcal P}}$ and
$\varphi\circ\psi=\mathrm{id}_{{\mathcal F}_{\kappa_{\bullet,\bullet}}(A_\bullet)}$.
Therefore, $\varphi$ is an isomorphism of algebraic varieties.
\end{proof}

To complete the proof of the theorem, it remains to construct a vector bundle
${\mathcal A}_x\cap{\mathcal P}\rightarrow{\mathcal A}_x^s\cap{\mathcal P}$.
To do this, we will use the following general result on parabolic orbits in partial
flag varieties for reductive groups.

\begin{lemma}
\label{lemma-vector-bundle}
Let $G$ be a connected, reductive algebraic group and let $P,Q\subset G$ be parabolic subgroups.
Let $L\subset P$ be a Levi subgroup, and let $S$ be the connected component of the center of $L$.
Let $U_P\subset P$ be the unipotent radical of $P$.
Let ${\mathcal P}\subset G/Q$ be a $P$-orbit, and consider the $S$-fixed point set ${\mathcal P}^S\subset {\mathcal P}$.
Fix $\xi \in {\mathcal P}^S$. \\
(a) The subvariety ${\mathcal P}^S$ is $L$-homogeneous, and $M:=\{\ell\in L:\ell \xi =\xi \}$ is a parabolic subgroup of $L$. 
So ${\mathcal P}^S=L/M$ is an irreducible, smooth, projective variety. \\
(b) The map $P=L\ltimes U_P\rightarrow L/M={\mathcal P}^S$
factorizes through the map $P\rightarrow {\mathcal P}$, $g\mapsto g\xi $ into an algebraic vector bundle
${\mathcal P}\rightarrow{\mathcal P}^S$.
\end{lemma}

A more detailed version of this result (giving for instance the dimension of the vector bundle)
together with an extension for loop groups
and affine flag varieties, can be found in \cite{Mitchell}.
For the sake of completeness, we give nonetheless the proof.

\begin{proof}
We start with an observation.
Let $\eta \in{\mathcal P}^S$ and $u\in U_P$, then the following points are equivalent:
(i) $u\eta \in{\mathcal P}^S$; (ii) $u\eta =\eta $. 
Indeed, (ii)$\Rightarrow$(i) is immediate on one hand.
To show the converse, fix a one-parameter subgroup $\lambda:\mathbb{K}^*\rightarrow S$ such that
\[
U_P=\{g\in G:\lim_{t\rightarrow 0}\lambda(t)g\lambda(t^{-1})=1_G\}.
\]
(cf. \cite[\S 8.4.5 and 8.4.6(5)]{Springer-book}.
Then, assuming that $u\eta \in{\mathcal P}^S$, we get $\lambda(t)u\lambda(t)^{-1}\eta =u\eta $ for all $t$, 
so passing to the limit as $t\rightarrow 0$ we infer that $u\eta =\eta $.

From this observation, it follows that the action of $L$ on ${\mathcal P}^S$ is transitive.
Indeed, let $\eta \in{\mathcal P}^S$. As ${\mathcal P}$ is a $P$-orbit, there is $g\in P$ such that $\eta =g\xi $.
Write $g=u\ell $ with $u\in U_P$ and $\ell \in L$, then we have $u(\ell \xi )=\eta \in{\mathcal P}^S$ hence $u(\ell \xi )=\ell \xi $.
Thus $\eta =\ell \xi $.
To complete the proof of (a), note that $\{g\in G:g\xi =\xi \}$ is a parabolic subgroup of $G$
containing $S$. Thus $M$, which is the intersection between $\{g\in G:g\xi =\xi \}$ and $L=Z_G(S)$,
is a parabolic subgroup of $L$.

Let $V\subset U_P$ denote the subset of elements which fix $\xi $. We claim that
$Fix_P(\xi ):=\{g\in P:g\xi =\xi \}=M\ltimes V$.
To show this, it is sufficient to show that, for $u\in U_P$, $\ell \in L$ such that
$u\ell \xi =\xi $, we have $u\xi =\xi $ and $\ell \xi =\xi $.
Since $\ell \xi ,u(\ell \xi )\in{\mathcal P}^S$, using the above observation, we infer that $u(\ell \xi )=\ell \xi $,
therefore $\ell \xi =\xi $. So, also, $u\xi =\xi $.

Since $P=L\ltimes U_P$ and $Fix_P(\xi )=M\ltimes V$, we obtain that ${\mathcal P}=P/Fix_P(\xi )=L\times_M(U_P/V)$.
In this fiber product, $M$ acts on $U_P$, $V$ by conjugation.
Let $\mathfrak{n}_P$ and $\mathfrak{v}$ denote the Lie algebras of $U_P$ and $V$ respectively, 
and consider the adjoint action of $M$ on them.
Through the exponential map, there is a $Q$-equivariant isomorphism
$\mathfrak{n}_P/\mathfrak{v}\stackrel{\sim}{\rightarrow} U_P/V$.
Therefore, the projection map ${\mathcal P}=L\times_M(\mathfrak{n}_P/\mathfrak{v})\rightarrow L/M={\mathcal P}^S$
is an algebraic vector bundle.
Whence (b).
\end{proof}

We apply Lemma \ref{lemma-vector-bundle} to our situation, with $G=GL(A)$,
$G/Q={\mathcal F}_{k_\bullet}(A)$, as above $P=\{g\in G:g\ker x^i=\ker x^i\ \forall i\}$,
$L=Z_G(s)$, and in particular we have ${\mathcal P}^s={\mathcal P}^S$.
The lemma yields that $\sigma:{\mathcal P}\rightarrow {\mathcal P}^s$, $(\ell u)\xi\mapsto \ell\xi$,
is an algebraic vector bundle.
In addition, as in the proof of the lemma, 
let $\mathfrak{n}_P$ and be the Lie algebra of $U_P$ and let $\mathfrak{v}=\{z\in \mathfrak{n}_P:\exp(z)\xi=\xi\}$.
Then, the fiber $\sigma^{-1}(\ell \xi)$ is isomorphic to the quotient $\mathfrak{n}_P/\mathfrak{v}$
via the map $z+\mathfrak{v}\mapsto \ell\exp(z)\xi$.

Consequently, the restriction $\sigma:\sigma^{-1}({\mathcal P}\cap{\mathcal A}_x^s)\rightarrow
{\mathcal P}\cap{\mathcal A}_x^s$ is also an algebraic vector bundle.
Observe that $1_G+x\in U_P$, and thus the group $X$ generated by $1_G+x$ is a finite subgroup of $U_P$.
It acts trivially on ${\mathcal P}\cap{\mathcal A}_x^s$ on one hand,
and its natural action on ${\mathcal P}$ leaves $\sigma^{-1}({\mathcal P}\cap{\mathcal A}_x^s)$
stable on the other hand.
The map $\sigma|_{\sigma^{-1}({\mathcal P}\cap{\mathcal A}_x^s)}$ is $X$-equivariant.
In addition, for $\ell\in L$ such that $\ell \xi\in {\mathcal P}\cap{\mathcal A}_x^s$
and for $z\in\mathfrak{n}_P$, we note that
\begin{eqnarray*}
(1_G+x)(\ell\exp(z)\xi) & = & \ell\exp\left(\mathrm{Ad}(\ell^{-1}(1_G+x)\ell)(z)\right)\ell^{-1}(1_G+x)^{-1}\ell\xi \\
 & = & \ell\exp\left(\mathrm{Ad}(\ell^{-1}(1_G+x)\ell)(z)\right)\xi.
\end{eqnarray*}
Thus, the action of $1_G+x$ induces a linear homomorphism 
$z+\mathfrak{v}\mapsto \mathrm{Ad}(\ell^{-1}(1_G+x)\ell)(z)+\mathfrak{v}$
on each fiber.
This shows that $\sigma:\sigma^{-1}({\mathcal P}\cap{\mathcal A}_x^s)\rightarrow
{\mathcal P}\cap{\mathcal A}_x^s$ is a $X$-vector bundle (cf. \cite[\S 1.6]{Atiyah}).
Note that ${\mathcal A}_x\cap{\mathcal P}$ is the $X$-fixed point set of $\sigma^{-1}({\mathcal P}\cap{\mathcal A}_x^s)$.
Applying \cite[\S 1.6]{Atiyah},
we infer that $\sigma$ restricts to a vector bundle ${\mathcal A}_x\rightarrow{\mathcal A}_x^s$.
This completes the proof of Theorem \ref{theorem1}.

\begin{remark}
The vector bundle $\sigma:{\mathcal P}\rightarrow{\mathcal P}^S$ can be interpreted 
as a homogeneization map in our particular setting, as follows.
For each $i$, we have a natural linear isomorphism $E_i\stackrel{\sim}{\rightarrow} \ker u^i/\ker u^{i-1}$,
then collecting these isomorphisms we get a linear isomorphism
$\pi:\bigoplus_{i=1}^m\ker u^i/\ker u^{i-1}\stackrel{\sim}{\rightarrow} \bigoplus_{i=1}^m E_i$.
For a subspace $V\subset A$, we set
\[
\textstyle{V^{(hom)}=\pi\left(\bigoplus_{i=1}^m(\ker u^i\cap V)/(\ker u^{i-1}\cap V)\right).}
\]
Thus, $V^{(hom)}\subset A$ is homogeneous and has the same dimension as $V$.
Hence,
\[
\hat\sigma:{\mathcal A}\rightarrow{\mathcal A}^s,\ (V_1,\ldots,V_n)\mapsto (V_1^{(hom)},\ldots,V_n^{(hom)})
\]
is well defined.
We observe that, if $V\subset A$ is a subspace,
\[
\dim V\cap \ker x^i=\dim (V\cap \ker x^i)^{(hom)}=\dim V^{(hom)}\cap \ker x^i.
\]
This implies that $\hat\sigma$ stabilizes the $P$-orbits of ${\mathcal A}$.
Moreover, if $V$ is stable by $x$, then $V^{(hom)}$ is stable by $x$, hence
$\hat\sigma({\mathcal A}_x)\subset{\mathcal A}_x^s$.
In fact, the vector bundle $\sigma$ coincides with the restriction
of the map $\hat\sigma$ to the orbit ${\mathcal P}$,
and the vector bundle ${\mathcal A}_x\cap {\mathcal P}\rightarrow{\mathcal A}_x^s\cap {\mathcal P}$
given in point (c) of the theorem is obtained as the restriction
of $\hat\sigma$ to ${\mathcal A}_x\cap {\mathcal P}$.
\end{remark}

\begin{remark}
\label{remark2}
Our construction can be related to
a construction by C. de Concini, G. Lusztig, C. Procesi
\cite{DeConcini-Lusztig-Procesi} of a partition of a Springer fiber (of any type)
into vector bundles over smooth, projective varieties.
In the latter construction, the partition is also obtained by taking the
intersections with the orbits of a parabolic subgroup naturally attached to
the nilpotent element. Each intersection is shown to be smooth, and to be a vector bundle
over its fixed-point set for the action of some semisimple element.
Nevertheless, the parabolic subgroup on which this construction relies
is obtained via the Jacobson-Morozov lemma, which leads to a partition
which does not coincide with ours. Our reasoning to get the structure of vector
bundle of the intersections with the parabolic orbits
is also different than in \cite{DeConcini-Lusztig-Procesi}.
In particular, in \cite{DeConcini-Lusztig-Procesi}, 
the structure of vector bundle is obtained by invoking
that the intersections are smooth, which is proved beforehand.
Here, we derive the smoothness of the intersections from their
structure of vector bundle, which we establish first.
\end{remark}

\end{document}